\documentclass[11pt,a4,leqno]{amsart}
\usepackage{amsthm}
\usepackage{amsmath}
\usepackage{amssymb}
\usepackage{amsfonts}
\usepackage{mathrsfs}
\usepackage{graphicx}
\usepackage{bbm}
\usepackage[
colorlinks=true, citecolor=blue, linkcolor=blue, urlcolor=red]{hyperref}
\usepackage{color}
\usepackage[all]{xy}
\usepackage{comment}
\usepackage{here}
\usepackage{amscd}
\usepackage{caption}
\usepackage{listings}
\usepackage{setspace}
\usepackage{enumitem}
\usepackage{tikz}
\usepackage{tikz-cd}
\allowdisplaybreaks

\newcommand{\mathsym}[1]{{}}
\newcommand{\unicode}[1]{{}}

\makeatletter

\@addtoreset{equation}{section}
\makeatother

\makeatletter
\@namedef{subjclassname@2020}{%
  \textup{2020} Mathematics Subject Classification}
\makeatother

\newtheoremstyle{my theoremstyle}
{1.0em}                    
    {1.0em}                    
    {\itshape}                   
    {}                          
    {\scshape}                  
    {.}                          
    {.5em}                      
    {}  

\newtheoremstyle{dfn}
{1.0em}                    
    {1.0em}                    
    {}                   
    {}                           
    {\scshape}                   
    {.}                          
    {.5em}                       
    {}  
    
\theoremstyle{my theoremstyle}
   \newtheorem{thm}{Theorem}[section]
   \newtheorem{lem}[thm]{Lemma}
   \newtheorem{prop}[thm]{Proposition}
   \newtheorem{cor}[thm]{Corollary}
   
\theoremstyle{dfn}

\theoremstyle{remark}   
   
   \newtheorem{rmk}[thm]{{\scshape Remark}}

\newcommand{\Z}{\mathbb{Z}}
\newcommand{\Q}{\mathbb{Q}}

\newcommand{\C}{\mathbb{C}}

\newcommand{\dv}{\operatorname{div}}

\numberwithin{equation}{section}

\newcommand{\lm}{\lambda}

\newcommand{\cer}{\operatorname{Cer}}
\newcommand{\ch}{\operatorname{CH}}
\newcommand{\pr}{\operatorname{pr}}
\newcommand{\ord}{\operatorname{ord}}
\newcommand{\jac}{\operatorname{Jac}}
\newcommand{\dvs}{\operatorname{div}}

\date{\today}

\begin{document}
\title[The modified diagonal cycles of Hypergeometric curves]{The modified diagonal cycles of Hypergeometric curves}
\author[Payman Eskandari]{Payman Eskandari$^*$}
\address{${}^*$Department of Mathematics and Statistics, University of Winnipeg, Winnipeg, Manitoba, Canada R3B 2E9.}
\email{p.eskandari@uwinnipeg.ca}
\email{y.nemoto@uwinnipeg.ca}
\author[Yusuke Nemoto]{Yusuke Nemoto$^{*\diamond}$}
\address{${}^{\diamond}$Graduate School of Science and Engineering, Chiba University, 
Yayoicho 1-33, Inage, Chiba, 263-8522 Japan.}
\email{y-nemoto@waseda.jp}
\date{\today}
\keywords{Ceresa cycle; modified diagonal cycle; Abel-Jacobi map; Hypergeometric curve.}
\subjclass[2020]{14C25, 14F25, 14H40}

\begin{abstract}
For each $N\geq 2$, Asakura and Otsubo have recently introduced a smooth family of algebraic curves $\{X_{N,\lambda}\}_{\lambda \in \mathbb{P}^1\setminus \{0, 1, \infty\}}$ in characteristic 0 that is closely related to hypergeometric functions and the Fermat curve of degree $N$. In this paper, we study the Gross-Kudla-Schoen modified diagonal 1-cycles of these curves. We prove that if $p \ge 3$ is a prime, then for every $\lambda$ the Griffiths Abel-Jacobi image of the modified diagonal cycle of $X_{p,\lambda}$ is nontrivial for every cuspidal choice of a base point. On the other hand, we show that the modified diagonal cycle of $X_{3,\lambda}$ is torsion in the Chow group for every $\lambda$ and cuspidal base divisor. 
\end{abstract}

\maketitle

\section{Introduction}\label{Introduction}
Let $X$ be a smooth projective curve over $\C$ with Jacobian $\jac(X)$. Let $e$ be a divisor of degree 1 on $X$. Denoting $\Delta=\{(x, x) \mid x \in X\} \in \ch_1(X^2)$ where $\ch$ is the Chow group\footnote{Unless otherwise indicated, all the Chow groups in the paper are with coefficients in $\Z$. All constructions of the paper involving algebraic cycles factor through Chow groups. Thus with abuse of terminology, we use the same notation and language for an algebraic cycle and its class in the Chow group.}, the modified diagonal cycle $\Delta_{{\rm GKS},e}(X)$ of Gross-Kudla-Schoen based at $e$ (first introduced by Gross and Schoen in \cite{GS}) is the homologically trivial algebraic 1-cycle
$$\Delta_{{\rm GKS}, e}(X):=\Delta_{123} - \Delta_{12} -\Delta_{23}-\Delta_{13} +\Delta_{1}+\Delta_{2} +\Delta_{3} \in \ch_1(X^3). $$
Here,
\[
\Delta_{123}=\{(x, x,x) \mid x \in X\},
\]
\begin{alignat*}{3}
&\Delta_{12}=\operatorname{pr}^*_{1,2}(\Delta) \cdot \operatorname{pr}_3^* (e) , \ &\Delta_{23}=\operatorname{pr}^*_{2,3}(\Delta)\cdot \operatorname{pr}_1^* (e), \ & \Delta_{13}=\operatorname{pr}^*_{1,3}(\Delta)\cdot \operatorname{pr}_2^* (e), \\
&\Delta_{1}=\operatorname{pr}_{2, 3}^* (e \times e), & \Delta_{2}=\operatorname{pr}_{1, 3}^* (e \times e), \  \quad \quad \ & \Delta_{3}=\operatorname{pr}_{1, 2}^* (e \times e), 
\end{alignat*}
where $\operatorname{pr}_i$ (resp. $\operatorname{pr}_{i, j}$) denotes the projection from $X^3$ onto the $i$-th factor (resp. the product of the $i$-th and the $j$-th factors) and $\cdot$ is the intersection product.
Another homologically trivial 1-cycle naturally associated with the pair $(X,e)$ is the Ceresa cycle (defined in \cite{Ceresa})
\[
\cer_e(X):= X_e-(-1)_\ast X_e \in \ch_1(\jac(X)),
\]
where $X_e$ is the image of $X$ under the map $X\rightarrow \jac(X)$ defined by $x \mapsto [x]-e$. 
Thanks to the works of Colombo and van Geemen \cite{CG}, S-W. Zhang \cite{Zhang}, Laga and Shnidman \cite{LS} and very recently Lagarde et al \cite{LMPRS}, the two cycles $\Delta_{{\rm GKS}, e}(X)$ and $\cer_e(X)$ are intimately related.
 In particular, one is torsion in the Chow group if and only if the other is. In addition, one knows the following two general facts about them: First, if $X$ is hyperelliptic and $e$ is supported on Weierstrass points, then the two cycles are torsion \cite{GS}. Second, if $X$ is a generic complex curve of genus $\geq 3$, then the two cycles are of infinite order for any $e$ \cite{Ceresa}.
\medskip\par 
Given a {\it specific} non-hyperelliptic curve $X$ of genus $\geq 3$, it can be quite difficult to prove anything nontrivial about the modified diagonal or the Ceresa cycle. We refer the reader to the works \cite{Harris} of Harris, \cite{Bloch} of Bloch, \cite{Kimura} of Kimura, \cite{T1}-\cite{T3} of Tadokoro, and \cite{Otsubo} of Otsubo for some results in this direction (see also \cite{EM2} for a survey of some of these results and the methods used). More recently, there have been several other developments on this topic using techniques different from those used in the prior work mentioned. These include developments in the following two directions:
\medskip\par 
\begin{itemize}[wide]
\item[(1)] {\it Infinitude of order in the Chow group}: Using the work of Darmon-Rotger-Sols \cite{DRS} on Chow-Heegner points, the first-named author and K. Murty showed in \cite{EM} that for every positive integer $N$ that has a prime divisor greater than 7, the Ceresa cycle of the Fermat curve $F(N)$ of degree $N$ is {\it non-torsion} in the Chow group for every choice of $e$. This work was generalized by the second author in \cite{Nemoto} to an infinite number of non-hyperelliptic simple Fermat quotients.\footnote{In \cite{EM} and \cite{Nemoto} $e$ was a point of the curve, but the proofs also work if $e$ is a divisor of degree 1.} Using the same method, Kerr-Li-Qiu-Yang \cite{KLQY} and Lupoian-Rawson \cite{LR} have proved that the Ceresa cycles of an infinite number of modular curves are non-torsion in the Chow group.
\medskip\par 
\item[(2)] {\it Non-hyperelliptic curves with torsion Ceresa cycles}: Beauville \cite{Beauville} has given an explicit non-hyperelliptic curve over $\Q$ for which the complex Abel-Jacobi image of the Ceresa cycle is torsion. Beauville's argument has been further developed by Beauville-Schoen \cite{BS}, Qui-Zhang \cite{QZ} and Laga-Shnidman \cite{LS2}. There are now several examples of non-hyperelliptic curves with torsion modified diagonal and Ceresa cycles in the Chow group (see the papers just cited), including an example of a family of such curves in genus 4 and such an example in genus 5 \cite{QZ}. Qui-Zhang \cite{QZ} and Laga-Shnidman \cite{LS2} have given sufficient cohomological criteria that would guarantee Ceresa cycles are torsion.\footnote{Related to both themes, there is also the work \cite{LS} of Laga and Shnidman, where they introduce a family of curves of genus 3 where they can determine exactly when the Ceresa cycle is torsion and when it is not torsion. Both phenomena occur in infinitely many fibers of their family.}
\end{itemize}
\medskip\par 
This paper fits in the two themes above. We consider the modified diagonal cycles of the family of complex algebraic curves
$$X_{N, \lm} : (x_2^N-x_1^N)(y_2^N-y_1^N)=\lm x_1^N y_1^N $$
in $\mathbb{P}^1\times \mathbb{P}^1=\{[x_1:x_2], [y_1:y_2]\}$ over the $\lambda$-line. This family has recently been introduced by Asakura and Otsubo \cite[Section 2.1]{AO}. Following them, we refer to this family as the family of {\it hypergeometric curves}, since the periods of $X_{N, \lm}$ are given by the values of hypergeometric functions \cite[Theorem 2.7]{AO}. The fiber $X_{N, 1}$ is the Fermat curve of degree $N$, and for $\lambda\in \mathbb{P}^1 \setminus \{0, 1, \infty\}$, $X_{N, \lm}$ is a smooth projective curve of genus $(N-1)^2$.
\medskip\par 
We prove two main results.
The first result is the following (see Theorem \ref{thm 1, more general} for a slightly more general statement):

\begin{thm} \label{main:1}
Let $p \ge3$ be a prime and $e \in X_{p, \lm}$ be a cusp (i.e., a point satisfying $x_1x_2y_1y_2=0$). 
Then for every $\lm \in \mathbb{P}^1 \setminus \{0, 1, \infty\}$ and positive integer $l \leq (p-1)/2$, we have 
$$l \cdot \Phi_1(\Delta_{{\rm GKS}, e}(X_{p, \lm})) \neq 0, $$
where $\Phi_1(\Delta_{{\rm GKS}, e}(X_{p, \lm}))$ is the complex Abel-Jacobi image of $\Delta_{{\rm GKS}, e}(X_{p, \lm})$ in the intermediate Jacobian of the integral Hodge structure $H_3(X_{p, \lambda}^3)$. In particular, the modified diagonal cycle  $\Delta_{{\rm GKS}, e}(X_{p, \lm})$ is nontrivial in the integral Chow group $\ch_1(X_{p, \lambda}^3)$ for every $\lambda\in \mathbb{P}^1 \setminus \{0, 1, \infty\}$.
\end{thm}

It has been observed in \cite[Proposition 2.2]{LS} that given any smooth projective complex curve $X$ of genus $g$ with canonical divisor $K(X)$, if $e$ is a degree 1 divisor on $X$ such that $(2g-2)[e]\neq K(X)$ in $\ch_0(X)\otimes\Q$, then the modified diagonal cycle of $X$ at $e$ is non-torsion in the Chow group. For a cuspidal choice of a base point $e$ as in Theorem \ref{main:1}, one has $(2(p-1)^2-2)[e]=K(X_{p,\lambda})$ in $\ch_0(X_{p,\lambda})\otimes\Q$ (see Corollary \ref{canonical}). Thus Theorem \ref{main:1} does not follow from \cite[Proposition 2.2]{LS}.

To our knowledge, Theorem \ref{main:1} provides the first explicit example of a smooth non-constant family $\mathscr{X}\rightarrow S$ of curves of arbitrarily large genus over a base $S$ where the modified diagonal cycle is nontrivial in the entire family. Since the hypergeometric curve $X_{p,\lm}$ is defined over $\Q(\lm)$, in particular, the result gives infinitely many curves of a given genus $(p-1)^2$ (for every prime $p\geq 3$) over any subfield of $\C$ with nontrivial modified diagonal cycles. By the recent result \cite[Proposition 5.6]{LMPRS}, Theorem \ref{main:1} also implies the nontriviality of the Ceresa cycles of hypergeometric curves $X_{p,\lambda}$ for $p\geq 5$ at cuspidal base points (see Example 5.15 in loc. cit.).

Unfortunately, we do not know whether $\Delta_{{\rm GKS}, e}(X_{p, \lm})$ is non-torsion or not for an arbitrary $p$.
However, for $p=3$ we can prove the following:

\begin{thm} \label{main:2}
For every $\lm \in \mathbb{P}^1 \setminus \{0, 1, \infty\}$ and every choice of a base divisor $e$ of degree 1 on $X_{3, \lm}$ supported on the cusps, the modified diagonal cycle $\Delta_{{\rm GKS},e}(X_{3, \lm})$ (and hence the Ceresa cycle $\cer_e(X_{3, \lm})$) is torsion in the Chow group.
\end{thm}

The curve $X_{3, \lm}$ is not hyperelliptic (see Proposition \ref{nonhyperelliptic}), so this gives an explicit example of a family of non-hyperelliptic curves of genus 4 with torsion modified diagonal cycles. Another example of such a family has been given by Qiu-Zhang \cite{QZ}. The proof of Theorem \ref{main:2} uses the cohomological vanishing criteria of Qiu-Zhang \cite[Theorem 1.2.1]{QZ} and Laga-Shnidman \cite[Theorem A]{LS2}.\footnote{We use the latter criterion to make the argument shorter, but the former criterion would also suffice in this case.} 
We expect that for $p\geq 5$ and a cuspidal divisor $e$, the cycle $\Delta_{{\rm GKS},e}(X_{p, \lm})$ is of infinite order for ``most" (or perhaps all) $\lambda$.

\subsection{About the proof of Theorem \ref{main:1}}\label{subsection: about the proofs} Here we say a few words about the proof of Theorem \ref{main:1}. In particular, we will explain why we only establish a lower bound on the order of the modified diagonal cycle, rather than prove its infinitude of order.
\medskip\par 
We follow the approach of \cite{EM}. Given a pointed smooth projective curve $(X,e)$, following an idea of Darmon-Rotger-Sols \cite{DRS}, one considers the pushforward of $\Delta_{{\rm GKS},e}(X)$ along a correspondence $Z\in\ch_2(X^4)$. This pushforward is a point in $\jac(X)$. 
When $X$ is the Fermat curve $F(p)$ for a prime number $p>7$ and $e$ is a cusp, for the choice
\[
Z= \Gamma_\alpha\times \Delta\in \ch_2(F(p)^4)
\]
where $\Delta\in \ch_1(F(p)^2)$ is the diagonal and $\Gamma_\alpha$ is the graph of the automorphism $\alpha \colon  (x:y:z)\mapsto (-y:z:x)$ of $F(p): x^p+y^p=z^p$, the pushforward of $\Delta_{{\rm GKS},e}$ turns out to be exactly the ``Gross-Rohrlich point" in $\jac(F(p))$ shown previously in \cite[Theorem 2.1]{GR} to be of infinite order, hence giving the result in \cite{EM}. 
\medskip\par 
To adapt this argument to the case of hypergeometric curves and prove (ideally) infinitude of order of $\Delta_{{\rm GKS},e}(X_{p,\lambda})$ for a prime $p$, we need a correspondence in $Z\in\ch_2(X_{p,\lambda}^4)$ that pushes $\Delta_{{\rm GKS},e}$ forward to a non-torsion point of $\jac(X_{p,\lambda})$. Otsubo has observed that for every $\lambda\in \mathbb{P}^1\setminus \{0,1,\infty\}$ the curve $X_{p,\lambda}$ has an automorphism of order 2 {\it different from} $(x,y)\mapsto (y,x)$, where $x=x_1/x_2$ and $y=y_1/y_2$ (see \S \ref{sec: proof of Thms 1 & 2}). Working with this automorphism, we define $Z$ similarly to the case of Fermat curves. One is then led to consider the pushforward of $\Delta_{{\rm GKS},e}(X_{p,\lm})$ along this correspondence, which is a point in $\jac(X_{p,\lambda})$. This is the analogue of the Gross-Rohrlich point in the previous argument. 
\medskip\par 
To complete this argument, we need an analogue of Gross-Rohrlich's \cite[Theorem 2.1]{GR} for hypergeometric curves. For this, we go back to the original proof of Gross and Rohrlich. The key ingredients of the argument for \cite[Theorem 2.1]{GR} are the following:
 \begin{itemize}[wide]
 \item[(1)] A careful analysis of the torsion part of $\jac(F(p))(\Q)$, which can be done very successfully using the zeta functions of Fermat curves (see \cite[Theorem 1.1]{GR}). 
 \item[(2)] The calculation of the space of meromorphic functions on $F(p)$ that are regular outside the cusps and whose possible poles at the cusps are at most of a given order (see \cite[Lemma 2.2]{GR}). 
\end{itemize} 
\medskip\par 
Unfortunately, the zeta functions of hypergeometric curves do not seem easy to understand, so we do not have (1) at our disposal. Instead, we prove a version of (2) (see Lemma \ref{key1}) that together with an elementary argument gives us a weak version of Gross and Rohrlich's \cite[Theorem 2.1]{GR} (see Proposition \ref{nontriv}) which is enough to prove Theorem \ref{main:1}. The lack of a good understanding of the torsion on the Jacobians of hypergeometric curves is the reason that our Theorem \ref{main:1} only establishes nontriviality with some lower bound on the order, rather than infinitude of order.

\subsection{Outline of the paper} We end this Introduction with an outline of the paper. In \S \ref{Background} below we recall some background and fix some notation. In \S \ref{sec: new results about HG curves} we prove certain new useful results about hypergeometric curves. We then prove Theorems \ref{main:1} and \ref{main:2} in \S \ref{sec: proof of Thms 1 & 2}.

\section{Preliminaries} \label{Background}
\subsection{Reminders on the Abel-Jacobi maps}
We start by recalling some basic constructions in Hodge theory (see \cite{Voisin} for a reference). 
Given a pure $\Z$-Hodge structure $H$, we denote by $H_{\Z}$ (resp. $H_{\C}$) the underlying free abelian group (resp. $H_\Z \otimes_{\Z} \C$). The notation $H^{\vee}$ denotes the dual of $H$ (as a Hodge structure).

Given a pure $\Z$-Hodge structure $H$ of odd weight $2k-1$, set
$$JH:=H_{\C}/(F^k H_{\C} + H_{\Z}).$$
Here, $F^{\bullet}$ is the Hodge filtration on $H_{\C}$. 
If $X$ is a smooth projective variety over $\C$, the cohomology group $H^n(X, \Z)$ underlies a pure $\Z$-Hodge structure of weight $n$, which we write as $H^n(X)$. Let
$$J_k(X):=J(H^{2k+1}(X)^{\vee}) \simeq (F^{k+1}H^{2k+1}(X, \C))^{\vee}/H_{2k+1}(X, \Z)$$
be the $k$-th Griffiths's intermediate Jacobian. Denoting the homologically trivial subgroup of $\ch_k(X)$ by $\ch_k(X)_{\rm hom}$, one has the Abel-Jacobi map
$$\Phi_k\colon \ch_k(X)_{\rm hom} \to J_k(X); \quad Z \mapsto \left(\eta \mapsto \int_{\Gamma} \eta \right)$$
for any $\eta \in F^{k+1}H^{2k+1}(X, \C)$, where $\Gamma$ is a topological $(2k+1)$-chain such that $\partial \Gamma=Z$.

\subsection{A criterion for finiteness of the order of the modified diagonal cycle}
Let $X$ be a smooth projective curve of genus $g \ge 3$ over $\C$. Let $\jac(X)$ and $K(X)$ be respectively the Jacobian and canonical divisor of $X$. Given a divisor $e$ of degree 1 on $X$, let
\[
\Delta_{{\rm GKS}, e}(X) \in \ch_1(X^3)_{\rm hom}
\]
be the modified diagonal cycle defined in \cite{GS} (see the \S \ref{Introduction} to recall the definition). We refer to this as the Gross-Kudla-Schoen modified diagonal cycle with base divisor (or point, if $e$ is a point) $e$. Gross and Schoen showed that this cycle is homologically trivial \cite[Proposition 3.1]{GS}. One also has the Ceresa cycle
\[
\cer_e(X):=X_e-(-1)_\ast X_e\in \ch_1(\jac(X))_{\rm hom}
\]
(see \S \ref{Introduction}). 
By \cite[Proposition 2.2]{LS}, if $\Delta_{{\rm GKS}, e}(X)$ (resp. $\cer_e(X)$) is torsion in the Chow group, then $(2g-2)e-K(X)$ is a torsion point of $\jac(X)\cong \ch_0(X)_{\mathrm{hom}}$. Thus the question of whether $\Delta_{{\rm GKS}, e}(X)$ (resp. $\cer_e(X)$) is torsion in the Chow group is reduced to the case when $(2g-2)e=K(X)$ in $\ch_0(X)_{\mathrm{hom}}\otimes\Q$. For such a base divisor $e$, by \cite[Theorem 1.5.5]{Zhang}, the finiteness of the orders of $\cer_e(X)$ and $\Delta_{{\rm GKS}, e}(X)$ in the corresponding Chow groups are equivalent. Qui and W. Zhang \cite[Theorem 1.2.1]{QZ} have given a sufficient cohomological criterion for this finiteness. This criterion has been strengthened by Laga and Shnidman in \cite[Theorem A]{LS2}. We recall the latter criterion, stated for $\Delta_{{\rm GKS}, e}(X)$ in terms of the de Rham cohomology:

\begin{thm}[Laga and Shnidman \cite{LS2}] \label{criteria}
Let $X$, $g$, and $K(X)$ be as above. Let $e$ be a divisor of degree 1 on $X$ such that $(2g-2)e=K(X)$ in $\ch_0(X)_{\mathrm{hom}}\otimes\Q$.
 Let $\operatorname{Aut}(X)$ be the automorphism group of $X$. Consider the action of $\operatorname{Aut}(X)$ on $\wedge^3 H^1_{\rm dR}(X)$ through its natural action on each factor of the de Rham cohomology $H^1_{\rm dR}(X)$, and let $(\wedge^3 H^1_{\rm dR}(X))^{\operatorname{Aut}(X)}$ be the fixed part under this action. 
If $(\wedge^3 H^1_{\rm dR}(X))^{\operatorname{Aut}(X)}=0$, then $\Delta_{{\rm GKS}, e}(X)$ is torsion in the Chow group.
\end{thm}
Note that \cite{LS2} works with Chow groups with coefficients in $\Q$, hence the assertion is stated as a vanishing statement in that paper. Also note that {\cite[Theorem A]{LS2}} is slightly stronger than the statement recorded above (and needed in this paper), as Laga and Shnidman only require the vanishing of the subspace of the \emph{primitive part} of $\wedge^3 H^1_{\rm dR}(X)=H^3_{\mathrm{dR}}(\jac(X))$ fixed by $\operatorname{Aut}(X)$.

\subsection{Hypergeometric curves}\label{more on hypergeometric curves}
 In this subsection we review the definition and some basic facts about the hypergeometric family of curves. All of what we will say in this subsection is due to Asakura and Otsubo \cite{AO}. The notation introduced in this subsection will be used throughout the paper.

Let $\mathbb{P}^1$ be the projective $\lm$-line over $\C$ and set $S=\mathbb{P}^1 \setminus \{0, 1, \infty\}$. 
Let $N \geq 2$ be a positive integer. 
Let
 $$f_N \colon \mathscr{X}_N \to S$$
 be the projective curve over $S$ defined in $\mathbb{P}^1 \times \mathbb{P}^1=\{([x_1 : x_2], [y_1: y_2])\}$ over $S$ by 
 \begin{align*}
 (x_2^N-x_1^N)(y_2^N-y_1^N)= \lm x_1^N y_1^N.
 \end{align*}
Its affine equation is 
 $$(1-x^N)(1-y^N)=\lm x^N y^N, $$
 where $x=x_1/x_2$, $y=y_1/y_2$. 
 For $\lm \in S$, let $X_{N, \lm}$ denote the fiber of $f_N$ over $\lm$. The curve $X_{N, \lm}$ is a smooth projective curve over $\C$.
Let $\mu_N$ be the group of $N$-th roots of unity. 
We set 
$$G_N=\Z/N\Z \times \Z/N\Z$$
and denote the element $(r, s) \in G_N$ by $g_N^{r, s}$. 
Fix a primitive $N$-th root of unity $\zeta_N$. 
 Then $G_N$ acts on $X_{N, \lm}$ as  
 $$g_N^{r, s} \cdot (x, y)= (\zeta_N^r x, \zeta_N^s y)  $$
 in the affine coordinates.
The character group $\operatorname{Hom} (G_N, \C^*)$ consists of 
$$\chi_N^{a, b}(g_N^{r, s}) = \zeta_N^{ar+bs} \hspace{.3in} (a,b\in \Z/N\Z).$$
Then one has the eigen decomposition 
$$H^1_{\rm dR}(X_{N, \lm})=  \bigoplus_{a, b=1}^{N-1} H^1_{\rm dR}(X_{N, \lm})^{\chi^{a, b}_N}, $$
where $H^1_{\rm dR}(X_{N, \lm})^{\chi^{a, b}_N}$ denote the submodule on which for every $r,s$, the element $g_N^{r, s}$ acts by multiplication by $\chi_N^{a, b}(g_N^{r, s})$ (see \cite[Proposition 2.4]{AO}).
For every $a,b\in \{1, \ldots, N-1 \}$, set
\begin{align*}
& \omega^{a, b}_{N, \lm} = N \dfrac{x^ay^b}{1-x^N} \dfrac{dx}{x} = -N \dfrac{x^ay^b}{1-y^N} \dfrac{dy}{y}, \\
& \eta^{a, b}_{N, \lm} = -\dfrac{b}{N \lm} (1-y^N) \omega_{N, \lm}^{a, b} = -\dfrac{b}{\lm} x^a y^b \dfrac{1-y^N}{1-x^N} \dfrac{dx}{x} = \dfrac{b}{\lm} x^a y^b \dfrac{dy}{y}. 
\end{align*}
The differential form $\omega^{a, b}_{N, \lm}$ is holomorphic on $X_{N,\lambda}$, and $\eta^{a, b}_{N, \lm}$ is non-holomorphic but of the second kind (= meromorphic with zero residues everywhere). The collection of differential forms $\omega^{a, b}_{N, \lm}$ and $\eta^{a, b}_{N, \lm}$ forms a basis of $H^1_{\rm dR}(X_{N, \lm})^{\chi_N^{a, b}}$. 

Given $a, b \in \{1, \ldots, N-1 \}$, let $C_{N, \lm}^{a, b}$ be the smooth projective curve defined in affine coordinates by
\begin{equation}\label{eq of the quotient curves}
v^N=(-u)^a (1-u)^{N-a} (1- \lm u)^{N-b}.
\end{equation}
There exists a finite morphism 
\begin{equation}\label{def of phi}
\varphi^{a, b}_N \colon X_{N, \lm} \to C_{N, \lm}^{a, b}; \quad (u, v)= \left(-\dfrac{x^N}{1-x^N}, \dfrac{x^a y^b}{(1-x^N)y^N} \right).
\end{equation}
Let $G^{a,b}_N$ be the subgroup of $G_N$ defined by
$$G^{a,b}_N= \{g_N^{r,s} \in G_N \mid ar+bs=0\}.$$
If $\gcd(N, a)=1$ or $\gcd(N,b)=1$, then $X_{N, \lm}$ is generically Galois over $C_{N, \lm}^{a, b}$ and
$$\operatorname{Gal}(X_{N, \lm}/C_{N, \lm}^{a, b}) =G^{a,b}_N$$
is cyclic of order $N$ generated by $g^{-a^{-1}b, 1}_N$ or $g^{1,-ab^{-1}}_N$. There is an automorphism of $X_{N, \lm}$ of order $2$ defined by $(x, y) \mapsto (y, x)$, 
which induces an isomorphism 
\begin{align}
C_{N, \lm}^{a, b} = X_{N, \lm} /G_N^{a, b} \simeq X_{N, \lm}/G_N^{b, a} = C_{N, \lm}^{b, a}. \label{ab}
\end{align}

For any integer $a$, let $\langle a \rangle \in \{0, \ldots, N-1\}$ be the representative of $a \pmod N$. It is convenient to set
\[
C_N^{a,b} := C_N^{\langle a\rangle ,\langle b\rangle}
\]
for every integers $a,b$ that are not multiples of $N$. If $j\in\{1, \ldots, N-1\}$ and $a$ or $b$ are coprime to $N$, then we have $G_N^{\langle a\rangle, \langle b\rangle }=G_N^{\langle aj\rangle, \langle bj\rangle }$, 
which induces an isomorphism 
\begin{align}
C_N^{a, b} \simeq C_N^{aj, bj}.  \label{isom}
\end{align}

\section{Some properties of hypergeometric curves}\label{sec: new results about HG curves}
The goal of this section is to establish some further results about hypergeometric curves.

\subsection{Determination of the hyperelliptic cases}\label{sec: characterization of hyperelliptic cases}
In this subsection we will determine when a hypergeometric curve  
is hyperelliptic. This is important to put Theorems \ref{main:1} and Theorem \ref{main:2} in a better context, since the behaviour of the modified diagonal cycle for hyperelliptic curves is well understood. We also determine when a quotient $C_{N, \lm}^{a, b}$ of a hypergeometric curve is hyperelliptic. This latter result will not be used in the proofs of Theorems \ref{main:1} and \ref{main:2}, however, it may be helpful for future study of modified diagonal cycles of these quotients. The result is also relevant to Proposition \ref{torsion} in \S \ref{subsection: further remarks}. 

\begin{prop} \label{nonhyperelliptic}
Let $N \ge 3$ be a positive integer. 
For any $\lm \in S$, the curve $X_{N, \lm}$ is non-hyperelliptic.  
\end{prop}

\begin{proof}
Suppose that $X_{N, \lm}$ is hyperelliptic. 
Let $\tau$ be its hyperelliptic involution. Since the genus of $X_{N, \lm}$ is greater than one, $\tau$ lies in the center of the automorphism group of $X_{N, \lm}$, so that we have $$\operatorname{Aut}(X_{N,\lm})/\langle \tau\rangle \subset   \operatorname{Aut}(X_{N,\lm}/\tau) =  \operatorname{Aut}(\mathbb{P}^1) = \operatorname{PGL}_2(\C).$$
First assume that $\tau \not \in \mu_N \times \mu_N$. 
Then we have $\mu_N \times \mu_N \subset \operatorname{Aut}(X_{N,\lm})/\langle \tau\rangle$.
By the well known characterization of the finite subgroups of $\operatorname{PGL}_2(\C)$ due to Klein, $\operatorname{Aut}(X_{N,\lm})/\langle \tau \rangle$ is isomorphic to one of the following: the cyclic group $C_n$ of order $n \ge 1$, the dihedral group $D_n$ of order $2n$ for $n \ge 2$, the alternating group $A_4$ or $A_5$, or the symmetric group $S_4$. However, these groups do not include any copy of the group $\mu_N \times \mu_N$. 

Next, we assume that $\tau \in \mu_N \times \mu_N$. 
Then $N$ is even and $\tau$ is given by one of the three elements of order 2 in $\mu_N \times \mu_N\subset \operatorname{Aut}(X_{N,\lambda})$, i.e., one of the following:
\begin{align*} 
\text{(i)}\ (x, y) \mapsto (-x, y), \quad \text{(ii)} \ (x, y) \mapsto (x, -y), \quad \text{(iii)} \ (x, y) \mapsto (-x, -y).  
\end{align*}
We only consider case (i), since the other cases are similar computations. In this case, $X_{N, \lm}/ \langle \tau \rangle$ is given by 
$$(1-x^{N/2}) (1-y^N)=\lm x^{N/2} y^N, $$
which has a genus equal to $(N/2-1)(N-1)$, which is not zero since $N \geq 3$.
\end{proof}

We now consider the quotients of $X_{N,\lambda}$.

\begin{prop} \label{hyperelliptic}
 Let $1\leq a,b\leq N-1$. Suppose $\gcd(N,a)=1$. Then the curve $C_{N, \lm}^{a, b}$ (defined by \eqref{eq of the quotient curves}) is hyperelliptic if and only if 
\begin{enumerate}
\item $a=b$ or $a+b=N$, or 
\item N=2m and b=m. 
\end{enumerate}
 \end{prop}
 
 \begin{proof}
First, we assume that $a=b$ or $a+b=N$. 
By \eqref{ab} and $\eqref{isom}$, it suffices to show the case $(a, b)=(1, 1)$ and ($1, N-1)$. 
There are isomorphisms
\begin{align*}
C_{N, \lm}^{1, 1} & \xrightarrow{\simeq} \{(z,w)\mid w^2=\dfrac{(1+\lm-z^N)^2}{4 \lm^2} - \dfrac1{\lm}\}\\
(u,v) &\mapsto (z,w)=\left(\dfrac{(1-u)(1-\lm u)}{v} \ , u-\frac{1+ \lm -z^N}{2 \lm} \right)
\end{align*}
and
\begin{align*}
C_{N, \lm}^{1, N-1} &\xrightarrow{\simeq} \{(z,w) \mid w^2=\dfrac{(1-z^N)^2}{4 \lm^2} + \dfrac{z^N}{\lm}\}\\
(u,v) &\mapsto (z,w)=\left(\dfrac{v}{1-u} \ , u-\frac{1- z^N}{2 \lm} \right),
\end{align*}
hence $C_{N, \lm}^{1, 1}$ and $C_{N, \lm}^{1, N-1}$ are hyperelliptic. 

Secondly, we assume that $N=2m$ and $b=m$. 
Then the curve $C_{2m}^{a, m}$ has an involution $(u, v) \mapsto (u, -v)$. The quotient curve by this involution is given by
$$v^m=(-u)^a (1-u)^{N-a} (1- \lm u)^{m}, $$
whose genus is equal to zero. 
Therefore, $C_{2m}^{a, m}$ is hyperelliptic.

\medskip\par 
Conversely, suppose that $C_{N, \lm}^{a, b}$ is hyperelliptic.
Let $\tau$ be its hyperelliptic involution. The automorphism group of $C_{N, \lm}^{a, b}$ contains the group $\overline{G_N^{a,b}}:=G_N/G_{N}^{a, b}$.
 
\underline{Case (i)}: Suppose that $\tau \not \in \overline{G_N^{a,b}}$. In this case we follow an argument of Coleman for an analogous result for quotients of Fermat curves \cite[Proposition 8]{Coleman}. Being in the center of the automorphism group, the action of $\tau$ on $C_{N, \lm}^{a, b}$ commutes with the action of $\overline{G_N^{a,b}}$. Thus $\tau$ acts on $C_{N, \lm}^{a, b}/\overline{G_N^{a,b}}$, which is the $u$-line. Moreover, the action of $\tau$ on the $u$-line is nontrivial since $\tau \not \in \overline{G_N^{a,b}}$. This action must permute the branch locus $\{0, 1, 1/\lm, \infty \}$. 
If $\lm \neq \frac1{\lm}, 1-\lm, \frac{\lm}{\lm-1}, \frac1{1-\lm}, \frac{\lm-1}{\lm}$, i.e. $\lm \neq -1, 1/2, 2, \zeta_6, \zeta_6^{-1}$, it follows that the action of $\tau$ on the $u$-line has to be one of
\begin{align} \label{tau}
\text{(i)}\ u \mapsto \dfrac{1-u}{1-\lm u}, \quad \text{(ii)} \ u \mapsto \dfrac{1- \lm u}{\lm(1-u)}, \quad \text{(iii)} \ u \mapsto \dfrac{1}{\lm u}. 
\end{align}
If $\lm =\frac1{\lm}$, i.e. $\lm = -1$, the possibilities of $\tau$ are \eqref{tau} and 
$$u \mapsto \frac1u, \quad  u \mapsto u \lm, \quad u \mapsto \dfrac{1-\lm u}{1- u}, \quad u \mapsto \dfrac{\lm(1-u)}{1- \lm u}. $$
If $\lm =1-\lm$, i.e. $\lm=1/2$, the possibilities of $\tau$ are \eqref{tau} and 
$$u \mapsto \dfrac{1- \lm u}{1-\lm }, \quad  u \mapsto \dfrac{u}{u-1}, \quad u \mapsto \dfrac{1}{1-\lm u}, \quad u \mapsto \dfrac{1-u}{(\lm-1)u}. $$
If $\lm = \frac{\lm}{\lm-1}$, i.e. $\lm =2$, the possibilities of $\tau$ are \eqref{tau} and 
$$u \mapsto1-u, \quad  u \mapsto \dfrac{(1-\lm)u}{1-\lm u}, \quad u \mapsto \dfrac{\lm-1}{\lm (1-u)}, \quad u \mapsto \dfrac{\lm u-1}{\lm u}. $$
If $\lm =\frac{\lm -1}{ \lm}$, $\frac1{1 - \lm}$, i.e. $\lm=\zeta_6, \zeta_6^{-1}$, the possibilities of $\tau$ are \eqref{tau} and 
\begin{align*}
&u \mapsto 1-\lm u, \quad  u \mapsto \dfrac{u-1}{u}, \quad u \mapsto \dfrac{1- \lm}{1-\lm u}, \quad u \mapsto \dfrac{(\lm-1)u}{1-u}, \\
&u \mapsto \dfrac{\lm(1-u)}{\lm-1 }, \quad  u \mapsto \dfrac{1}{1-u}, \quad u \mapsto \dfrac{1-\lm u }{(1-\lm) u}, \quad u \mapsto \dfrac{ \lm u}{\lm u-1}.  
\end{align*}
The automorphism of the $u$-line given by $\tau$ must preserve the equation
$v=(-u)^a (1-u)^{N-a} (1- \lm u)^{N-b}$ of $C_{N, \lm}^{a, b}/\overline{G_N^{a,b}}$. Each case above now implies that $a=b$ or $a+b=N$. 

\underline{Case (ii)}: Suppose that $\tau \in \overline{G_N^{a,b}}$. 
Since $\overline{G_N^{a,b}} \simeq \Z/N\Z$, it follows that $N=2m$ and $\tau(u, v)=(u, -v)$. 
The quotient curve $C_{N, \lm}^{a,b}/\langle \tau \rangle$ is then given by 
$$v^m=(-u)^a(1-u)^{2m-a}(1- \lm u)^{2m-b}, $$
whose genus must be equal to zero. 
Therefore, we have $m \mid b$, i.e. $b=m$, which finishes the proof. 
\end{proof}

\subsection{Properties of the cusps}\label{sec: results about cusps}
Throughout, $\lambda\in \mathbb{P}^1\setminus\{0,1,\infty\}$. From now on we let $\rho$ be a fixed $N$-th root of $1-\lm$ in $\C$. 
On $X_{N, \lm}$, we have $4N$ points
\begin{alignat*}{2}
&a_i=\left([0 :1], [\zeta_N^i:1] \right), &&\quad b_i=\left([\zeta_N^i:1], [0 :1] \right),\notag \\
&c_{1, i}= ([1 : 0], [\rho^{-1} \zeta_N^i : 1]), &&\quad c_{2, i}= ([\rho^{-1} \zeta_N^i: 1], [1:0]),
\end{alignat*}
where $i=0, \ldots, N-1$. We call these $4N$ points the {\it cusp} points of $X_{N, \lm}$ and refer to the other points of $X_{N, \lm}$ as {\it affine} points. Note that the function $x$ (resp. $y$) has its zeros at the $a_i$ (resp. $b_i$) and its poles at the $c_{1,i}$ (resp. $c_{2,i}$).

Our goal in this subsection is to establish some results about the cusp points. These results will be needed in the proofs of Theorems \ref{main:1} and \ref{main:2}, which will be given in \S \ref{sec: proof of Thms 1 & 2}. 

We start with an analogue of the Manin-Drinfeld theorem for hypergeometric curves:
\begin{prop} \label{torsion2}
Let $A$ and $B$ be cusps on $X_{N, \lm}$.   
Then the point $[A]-[B] \in \jac(X_{N, \lm})$ is a torsion point of order dividing $N^2$. If $\{A,B\}$ is any of $\{b_i, c_{2, j}\}$, $\{a_i,c_{1, j}\}$, $\{a_i, a_j\}$, $\{b_i, b_j\}$, $\{c_{1, i}, c_{1, j}\}$ or $\{c_{2, i}, c_{2, j}\}$ for any $i,j$ then $[A]-[B]$ is annihilated by $N$.
\end{prop}

\begin{proof}
For $i=0, \ldots, N-1$, we have 
\begin{alignat*}{2} 
&\dvs (x-\zeta_N^i)= N b_i - \sum_{j=0}^{N-1} c_{1, j} \ , \quad &&\dvs (x- \rho^{-1} \zeta_N^i)= N c_{2, i} - \sum_{j=0}^{N-1}  c_{1, j}, \\
&\dvs (y-\zeta_N^i)= N a_i - \sum_{j=0}^{N-1}  c_{2, j} \ , 
\quad &&\dvs (y- \rho^{-1} \zeta_N^i)= N c_{1, i} - \sum_{j=0}^{N-1}  c_{2, j}.  
\end{alignat*}
Therefore, for any $i, j \in \{0, \ldots, N -1\}$, we have 
\begin{align}
&\dvs \left(\frac{x-\zeta_N^i}{x-\rho^{-1} \zeta^j_N}\right)= N (b_i - c_{2, j}), \quad \dvs \left(\frac{y-\zeta_N^i}{y-\rho^{-1} \zeta^j_N}\right)= N (a_i - c_{1, j}), \label{tor1}
\end{align}
so that $[b_i]-[c_{2, j}]$ and $[a_i]-[c_{1, j}]$ are annihilated by $N$. 
Therefore, $[a_i]-[a_j]$ and $[b_i]-[b_j]$ as well as $[c_{1, i}]-[c_{1, j}]$ and $[c_{2, i}]-[c_{2, j}]$ are also killed by $N$. 
On the other hand, we have 
\begin{align}
\dvs \left(\dfrac{(x-\zeta_N^i)^N (y^N-\rho^{-N})}{(y-\zeta_N^j)^N} \right)=N^2(b_i -a_j),  \label{tor2}
\end{align}
hence $[b_i]-[a_j]$ is killed by $N^2$. By \eqref{tor1} and \eqref{tor2}, we conclude that the remaining differences of two cusps are also killed by $N^2$.
\end{proof}

\begin{cor} \label{canonical}
Let $K(X_{N, \lm})$ be the canonical divisor of $X_{N, \lm}$. 
For every divisor $e$ of degree $1$ on $X_{N, \lm}$ supported on the cusps, we have
$$(2(N -1)^2 - 2)e = K(X_{N, \lm})$$
in $\ch_0(X_{N, \lm}) \otimes \Q$. 
\end{cor}

\begin{proof}
We have
$$K(X_{N, \lm})=\text{div}(dx)=(N-1)\sum_{i=0}^{N-1} ([b_i]+[c_{2, i}])-2\sum_{i=0}^{N-1} [c_{1, i}]. $$
Thus by Proposition \ref{torsion2}, $K(X_{N, \lm}) - (2(N -1)^2 - 2)e $ is torsion in the Jacobian, which proves the assertion. 
\end{proof}

Our final result in this subsection is a result about the space of meromorphic functions with bounded poles at the cusps. The result may seem rather unmotivated by itself, but it will play a key role in the proof of Theorem \ref{main:1}.
For a positive integer $d$ and distinct points $e_0, \ldots, e_n \in X_{N, \lm}$, 
let
$$\mathscr{L}\left(d \displaystyle \sum_{i=0}^{n} [e_i]\right)$$
be the Riemann-Roch space of the divisor $d \displaystyle \sum_{i=0}^{n} [e_i]$, i.e., 
 the vector space (over the constants $\C$) of meromorphic functions on $X_{N,\lm}$ that are regular outside $\{e_i\mid i=0, \ldots, n\}$ and do not have a pole of order greater than $d$ at any $e_i$.
 
\begin{lem} \label{key1} 
Let $N=p$ be a prime number. Let $d$ be an integer such that $0\leq d \leq p-1$. Then we have the following statements:
\begin{enumerate}  
\item A basis for $\mathscr{L}\left(d \displaystyle \sum_{i=0}^{p-1} [c_{1, i}] \right)$ is given by $\{x^m \mid 0 \leq m \leq d\}$. 
\item A basis for $\mathscr{L}\left(d \displaystyle \sum_{i=0}^{p-1} [c_{2, i}] \right)$ is given by $\{y^m \mid 0 \leq m \leq d\}$.
\item A basis for $\mathscr{L}\left(d \displaystyle \sum_{i=0}^{p-1} [a_{i}] \right)$ is given by $\{x^{-m} \mid 0 \leq m \leq d\}$.
\item A basis for $\mathscr{L}\left(d \displaystyle \sum_{i=0}^{p-1} [b_{i}] \right)$ is given by $\{y^{-m} \mid 0 \leq m \leq d\}$.
\end{enumerate}
\end{lem}

\begin{proof}
We only include the proof of (i) since the other statements are proved similarly. The first part of the proof is similar to Gross-Rohrlich's argument for \cite[Lemma 2.2]{GR}.
Let $K$ be the function field of $X_{p, \lm}$ and let $k$ be the subfield generated over the constants by $x^p$. 
Then $\{x^m y^n \mid 0 \leq m, n \leq p-1\}$ is a basis of $K$ over $k$ (recall that $\lambda\neq 0,1$). 
Suppose that 
$$f=\sum_{m, n} f_{m, n} x^m y^n, \quad f_{m, n} \in k \quad (0\leq m,n\leq p-1)$$ 
is regular outside the $c_{1, i}$ $(i=0, \ldots, p-1)$. We claim that for any $m, n \in \{0, \ldots, p-1\}$, $f_{m ,n}$ is a {\it polynomial} in $x^p$. Indeed, let $\operatorname{Tr}^K_k$ be the trace map from $K$ to $k$. 
Then we compute  
\begin{align*}
\operatorname{Tr}^K_k(x^my^n)&=\sum_{g_p^{a, b} \in G_p} g^{a, b}_p \cdot x^m y^n= \sum_{a,b\in \Z/p\Z} \zeta_p^{am+bn} x^m y^n \\
&=\left\{
\begin{array}{ll}
p^2 & \text{if} \ (m, n) \equiv (0, 0) \pmod{p}\\
0 & \text{otherwise},
\end{array}
\right.
\end{align*}
so that we have 
$$\operatorname{Tr}^K_k(f x^{-m}y^{-n})=p^2 f_{m, n}. $$
Thus we may show that $\operatorname{Tr}^K_k(f x^{-m}y^{-n})$ is a polynomial in $x^p$. Since $f x^{-m} y^{-n}$ is regular at any affine point of $X_{p, \lm}$, so is $\operatorname{Tr}^K_k(f x^{-m}y^{-n})$. On recalling the behavior of $f, x$ and $y$ at the $a_i$ and $b_i$, we have 
\begin{align*}
\ord_{a_i} (f x^{-m} y^{-n}) \geq -m > -p, \quad \ord_{b_i} (f x^{-m} y^{-n}) \geq -n > -p, 
\end{align*}
so that
\begin{align*}
\ord_{a_i} \operatorname{Tr}^K_k(f x^{-m}y^{-n}) >  -p,  \quad \ord_{b_i} \operatorname{Tr}^K_k(f x^{-m}y^{-n}) > -p.  
\end{align*}
Since $K$ is ramified over $k$ of order $p$ at $a_i$'s and $b_i$'s, we conclude that 
$$\ord_{a_i} \operatorname{Tr}^K_k(f x^{-m}y^{-n}) \geq 0, \quad \ord_{b_i} \operatorname{Tr}^K_k(f x^{-m}y^{-n}) \geq 0. $$
Furthermore, for any $i \in \{0, \ldots, p-1\}$, we have 
$$\ord_{c_{2, i}} (f x^{-m}y^{-n}) \geq n \geq 0, $$
hence $\ord_{c_{2, i}} \operatorname{Tr}^K_k(f x^{-m}y^{-n}) \geq 0$. 
Therefore, $\operatorname{Tr}^K_k(f x^{-m}y^{-n})$ (which is a priori a rational function in $x^p$) is regular on $X_{N,\lambda}$ outside the $c_{1, i}$'s. It follows that $\operatorname{Tr}^K_k(f x^{-m}y^{-n})$ is a polynomial in $x^p$, as claimed.

Now suppose that $f \in \mathscr{L}\left(d \displaystyle \sum_{i=0}^{p-1} [c_{1, i}]\right)$. The fact that the $f_{m,n}$ are polynomials in $x^p$ together with the bound on the orders of the poles of $f$ as the $c_{1,i}$ implies that the $f_{m,n}$ must be constants. That is, $f$ has a unique expression 
\begin{align*}
f= \sum a_{m, n} x^m y^n, \quad a_{m, n} \in \C, 
\end{align*}
where the sum is taken over $0 \leq m, n \leq p-1$.
We first show that $a_{m,n}$ is zero if $m>d$. To see this,
tentatively, let $e\geq 0$ be the largest integer such that $a_{e, n} \neq 0$ for some $n$. 
Write
$$f=\sum_{m=0}^{e}\left(\sum_{n = 0}^{p-1} a_{m, n} y^n \right) x^m. $$
Suppose $e>d$. Since $f$ has a pole of order at most $d$ at each $c_{1, i}$, for each $i \in \{0, \ldots, p-1\}$ we must have
 $$\sum_{n = 0}^{p-1} a_{e, n} (\rho^{-1} \zeta_p^i)^n =0. $$
 This is absurd because $\sum_{n = 0}^{p-1} a_{e, n} t^n\in \C[t]$ cannot have $p$ distinct roots in $\C$. 

To complete the argument, it remains to show that $a_{m, n}=0$ if $n > 0$. This is done by a similar argument as the last one: Take $e$ now to be the largest integer such that $a_{m, e}\neq 0$ for some $m$. Assuming $e>0$, since $f$ does not have poles at the $c_{2, i}$, for each $i=0, \ldots, p-1$ we get
 $$\sum_{m = 0}^{d} a_{m, e} (\rho^{-1} \zeta_p^i)^m =0, $$
which is absurd since $d\leq p-1$.
\end{proof}

\section{Proof of Theorems \ref{main:1} and \ref{main:2}}\label{sec: proof of Thms 1 & 2}

\subsection{Proof of Theorem \ref{main:1}}\label{sec: proof of thm 1}
Our goal in this subsection is to prove Theorem \ref{main:1}.
As before, let $\lambda\in \mathbb{P}^1\setminus\{0,1,\infty\}$ and $\rho$ a fixed $N$-th root of $1-\lambda$ in $\C$. The curve $X_{N, \lm}$ has an automorphism 
$$\alpha \colon X_{N, \lm} \to X_{N, \lm}; \quad (x, y) \mapsto \left(\frac1{\rho x}, \frac1{\rho y} \right), $$
which has order 2. If $N$ is odd, then $\alpha$ has exactly two fixed points 
$$P=\left({\xi}, -{\xi} \right), \quad Q=\left(-{\xi}, {\xi} \right),  $$
where $\xi^2=\rho^{-1}$.

Recall from \S \ref{more on hypergeometric curves} that $\varphi_N^{a,b}$ is the map \eqref{def of phi} from $X_{N,\lambda}$ to the quotient $C_{N,\lambda}^{a,b}$.

\begin{prop} \label{nontriv}
Let $p \neq 2$ be a prime number and $l$ a positive integer such that $l \leq (p-1)/2$. Let $\lambda\in \mathbb{P}^1\setminus\{0,1,\infty\}$ and $e$ be a divisor of degree 1 on $X_{p,\lambda}$ supported on one of the four sets of cusps $\{c_{1,i}: 0 \leq i<p\}$, $\{c_{2,i}: 0 \leq i<p\}$, $\{a_i: 0 \leq i<p\}$, or $\{b_i: 0 \leq i<p\}$. Then the point 
$$l \cdot (\varphi^{a, b}_p)_*([P] + [Q] -2e) \in \jac(C_{p, \lm}^{a, b})$$
 is nontrivial for every $a,b\in \{1,\ldots, p-1\}$.
\end{prop}

\begin{proof}
We only prove the case when $e$ is supported on the $c_{1,i}$ since the other cases are proved similarly. It suffices to show that the point
$$l \cdot (\varphi_p^{a, b})^* \circ (\varphi^{a, b}_p)_*([P] + [Q] -2e)=l \cdot \sum_{j=0}^{p-1} \left(g_p^{bj, -aj} \cdot ([P] + [Q] ) - 2 [c_{1, j}] \right)$$
of $\jac(X_{p,\lambda})$ is nontrivial. 
Here, the action of $g_p^{bj, -aj}$ in the right hand side is by pushforward of divisors. The equation uses the fact that $e$ has degree 1, as well as the fact that for every $0\leq i<p$, we have $(\varphi_p^{a, b})^* \circ (\varphi^{a, b}_p)_*([c_{1,i}]) = \sum_{j=0}^{p-1} [c_{1, j}]$. 

Let $K$ be the function field of $X_{p, \lm}$.
Suppose that the above point of $\jac(X_{p,\lambda})$ is trivial. Then there exists a function $f \in K$ such that 
$$\operatorname{div} (f)=  l \cdot \sum_{j=0}^{p-1} \left(g_p^{bj, -aj}  \cdot ([P] + [Q] ) - 2 [c_{1, j}] \right). $$
Since $f$ has poles of order $2l \leq p-1$ at $c_{1, j}$, by Lemma \ref{key1}, we can write 
$$f=\sum_{0\leq m \leq 2l} a_m x^m \quad \quad (a_m\in\C).$$ 
On the other hand, $f$ vanishes at $P$ and $Q$, hence $f$ is divisible by the polynomial $x^{2} -\rho^{-1}$ in the polynomial ring $\C[x]$, i.e., there exists a polynomial $g \in \C[x]\subset K$ such that 
$$f=(x^{2} -\rho^{-1})g. $$
The divisor of $x^2-\rho^{-1}$ is given by 
$$\operatorname{div}(x^{2}- \rho^{-1})= \sum_{j=0}^{p-1} \left(g_p^{0, j}\cdot  ([P] + [Q]) - 2[c_{1, j}] \right), $$
hence we have 
\begin{align} 
\begin{split}
\dvs(g) &=(l-1)([P]+[Q]) + l \sum_{j=1}^{p-1} g_p^{bj, -aj}  \cdot ([P]+ [Q]) \\
&- 2(l-1) \sum_{j=0}^{p-1} [c_{1, j}] -\sum_{j=1}^{p-1} g_p^{0, j} \cdot ([P]+[Q]).  
\end{split}\label{g}
\end{align}
It follows that $g$ has poles at $g_p^{0, j}(P)$ and $g_p^{0, j}(Q)$ for $0<j<p$ (note that these are not canceled by the zeros of $g$; also note that these are distinct from the points $c_{1, j}$ for $0\leq j<p$). But this is absurd, since $g$ is a polynomial in only $x$ and therefore it does not have any poles outside the $c_{1, j}$.
 \end{proof}

We will now deduce Theorem \ref{main:1}. 
In fact, we may allow the base $e$ to be slightly more general:
\begin{thm}\label{thm 1, more general}
Let $p \ge3$ be a prime, $\lm \in \mathbb{P}^1 \setminus \{0, 1, \infty\}$, and $e$ a divisor of $X_{p, \lm}$ of degree 1 with support contained in one of the four sets of cusps $\{c_{1,i}: 0\leq i<p\}$, $\{c_{2,i}:0\leq i<p\}$, $\{a_{i}: 0\leq i<p\}$, or $\{b_{i}: 0\leq i<p\}$. Then for every positive integer $l \leq (p-1)/2$, we have 
$$l \cdot \Phi_1(\Delta_{{\rm GKS}, e}(X_{p, \lm})) \neq 0, $$
where $\Phi_1(\Delta_{{\rm GKS}, e}(X_{p, \lm}))$ is the complex Abel-Jacobi image of $\Delta_{{\rm GKS}, e}(X_{p, \lm})$ in the intermediate Jacobian of the integral Hodge structure $H_3(X_{p, \lambda}^3)$.
\end{thm}
\begin{proof}

For brevity, we set $X=X_{p, \lm}$.
We take $\Gamma_{\alpha} \in \ch_1(X \times X)$ to be the graph of $\alpha$.
Let $\pr_{123} \colon X^3 \times X \to X^3$ (resp. $\pr_4 \colon X^3 \times X \to X$) be the natural projection to the first three factors of $X$ (resp. the fourth factor of $X$). 
Let $\Delta \in \ch_1(X \times X)$ be the diagonal cycle.  
Then the correspondence 
$$Z=\Gamma_{\alpha} \times \Delta \in \ch_2(X^2 \times X^2)=\ch_2(X^3 \times X)$$
gives a map
$$\Pi_{Z} \colon \ch_1(X^3)_{\rm hom} \to \ch_0(X)_{\rm hom}; \quad \Omega \mapsto (\pr_4)_*(\pr_{123}^*(\Omega) \cdot (\Gamma_{\alpha} \times \Delta).$$
Thanks to the functoriality of Abel-Jacobi maps, the map $\Pi_Z$ fits into a commutative diagram
\begin{equation}\label{diagram for pf of Thm 1}
\begin{tikzcd}
    \ch_1(X^3)_{\rm hom} \arrow[r, "\Pi_Z"] \arrow[d] & \ch_0(X)_{\rm hom} \arrow[d] \\
    J_1(X^3) \arrow[r] & J_0(X),
\end{tikzcd}
\end{equation}
where the vertical maps are the Abel-Jacobi maps and the bottom horizontal map is induced by the Hodge class of $\Pi_Z$. A straightforward computation shows
\begin{align*} \label{DRS}
\Pi_{Z}(\Delta_{{\rm GKS}, e}(X))= [P] + [Q] -2\alpha_*(e). 
\end{align*}
For a divisor $e$ with support in one of the four sets $\{c_{1,i}\}$, $\{c_{2,i}\}$, $\{a_{i}\}$, or $\{b_{i}\}$, the support of $\alpha_*(e)$ is also in one of the same sets.
Thus $l \cdot \Pi_{Z}(\Delta_{{\rm GKS}, e}(X))$ is nontrivial for $l\leq (p-1)/2$ by Proposition \ref{nontriv}. The theorem now follows from the commutative diagram \eqref{diagram for pf of Thm 1} since the Abel-Jacobi map on the right side of the diagram is an isomorphism.
\end{proof}
\subsection{Proof of Theorem \ref{main:2}}
We now prove Theorem \ref{main:2}. 

The automorphism group $\operatorname{Aut}(X_{3,\lambda})$ acts naturally on $\wedge^3 H^1_{\rm dR}(X_{3,\lambda})$. For a subgroup $G \subset \operatorname{Aut} (X_{3,\lambda})$, we write the fixed part of $G$ for this action as $(\wedge^3H^1_{\rm dR}(X_{3,\lambda}))^G$.  
Since $G_3=\Z/3\Z\times \Z/3\Z \subset {\operatorname{Aut}(X_{3,\lambda})}$, 
we have 
$$(\wedge^3H^1_{\rm dR}(X_{3,\lambda}))^{\operatorname{Aut}(X_{3,\lambda})} \subset (\wedge^3H^1_{\rm dR}(X_{3,\lambda}))^{G_3}, $$
hence by Corollary \ref{canonical} and the criterion of Theorem \ref{criteria}, to prove Theorem \ref{main:2} it suffices to show that $(\wedge^3H^1_{\rm dR}(X_{3,\lambda}))^{G_3}=0$. Recalling the decomposition of $H^1_{\rm dR}(X_{3,\lambda})$ according to the characters of $G_3$ from \S \ref{more on hypergeometric curves}, we have
\begin{align*}
&(\wedge^3 H^1_{\rm dR}(X_{3,\lambda}))^{G_3} 
= \bigoplus_{\substack{a_i, b_i=1, 2 \\ \sum_i a_i=\sum_i b_i=0}}
H^1_{\rm dR}(X_{3,\lambda})^{\chi_3^{a_1, b_1}} \wedge H^1_{\rm dR}(X_{3,\lambda})^{\chi_3^{a_2, b_2}}  \wedge H^1_{\rm dR}(X_{3,\lambda})^{\chi_3^{a_3, b_3}}  \\
&=
\langle \varphi_1^{1, 1} \wedge \varphi_2^{1, 1} \wedge \varphi_3^{1, 1}, \varphi_1^{2, 2} \wedge \varphi_2^{2, 2} \wedge \varphi_3^{2, 2}, \varphi_1^{1, 2} \wedge \varphi_2^{1, 2} \wedge \varphi_3^{1, 2}, \varphi_1^{2, 1} \wedge \varphi_2^{2, 1} \wedge \varphi_3^{2, 1}\rangle, 
\end{align*}
where $\varphi^{a, b}_i$ is $\omega_{3, \lm}^{a, b}$ or $\eta_{3, \lm}^{a, b}$ for each $i$. It follows that $(\wedge^3 H^1_{\rm dR}(X_{3,\lambda}))^{G_3}$ is indeed zero.

\begin{rmk}
For $p\geq 5$, it is not difficult to write an explicit nonzero element of $H^3_{\rm dR}(\jac(X_{p,\lambda}))$ (in fact, an element of the primitive part) that is fixed by $\Z/p\Z\times \Z/p\Z$, the automorphism $\alpha$, and the automorphism given by $(x,y)\mapsto (y,x)$.
\end{rmk}

\subsection{Some further remarks}\label{subsection: further remarks}
Here we make a few supplementary comments on Theorem \ref{main:1}. Let $p$ be a prime number $>3$ and $e$ a cusp of $X_{p,\lambda}$. If the image of the point $[P]+[Q]-2[e]$ in the Jacobian of one quotient curve $C_{p,\lm}^{a,b}$ 
is non-torsion, then the argument in \S \ref{sec: proof of thm 1} together with Proposition \ref{torsion2} implies that $\Delta_{{\rm GKS}, e'} (X_{p, \lm})$ is non-torsion for every choice of
a degree 1 base divisor $e'$. 
We do not know how to prove that $(\varphi_{p}^{a,b})_\ast([P]+[Q]-2[e])$ is non-torsion for any $a,b$. 
However, in the cases when $(a, b)$ is $(1, 1)$ or $(1, p-1)$ (which are the hyperelliptic cases, see Proposition \ref{hyperelliptic}), we have the following:

\begin{prop} \label{torsion}
Let $p$ be a prime number and $(a, b)\in\{(1, 1),(1, p-1)\}$.
Then for any divisor $e$ of degree 1 on $X_{p, \lm}$ supported on the cusps, the point 
$$(\varphi^{a, b}_p)_* ([P]+[Q]-2e) \in \jac(C_{p, \lm}^{a, b})$$
is torsion for every $\lm \in \mathbb{P}^1\setminus\{0,1,\infty\}$.\footnote{Note this it is not clear whether this statement can be deduced from the torsion-ness of the modified diagonal cycles of hyperelliptic curves. Indeed, it is not obvious whether the point $(\varphi^{a, b}_p)_* ([P]+[Q]-2e)$ of $\jac(C_{p, \lm}^{a, b})$ is of the form $\Pi_{Z}(\Delta_{{\rm GKS}, e}(C_{p, \lm}^{a, b}))$ for some algebraic cycle $Z$ on $(C_{p, \lm}^{a, b})^4$.}
\end{prop}

\begin{proof}
By Proposition \ref{torsion},  it suffices to show the case $e=[c_{1, 0}]$. 
Since 
\begin{align*}
&p \cdot (\varphi^{1, 1}_p)^* \circ  (\varphi^{1, 1}_p)_* ([P]+[Q]-2[c_{1, 0}]) = \dv ((\rho^{-1} + xy)(1+x^p)), \\ 
&p \cdot (\varphi^{1, p-1}_p)^* \circ (\varphi^{1, p-1}_p)_* ([P]+[Q]-2[c_{1, 0}]) = \dv ((x-\rho^{-1}y)(1+x^p)),  
\end{align*}
we have $(\varphi^{a, b}_p)^* \circ (\varphi^{a, b}_p)_* ([P]+[Q]-2[c_{1, 0}])=0$ in $\jac(X_{p, \lm}) \otimes \Q$ for $(a, b)=(1, 1), (1, p-1)$. The proposition follows since the map
$$(\varphi^{a, b}_p)^*  \colon \jac(C_{p, \lm}^{a, b}) \otimes \Q \to \jac(X_{p, \lm}) \otimes \Q$$
is injective. 
\end{proof}

We end the paper with a comment about the case of the curves $X_{N,\lm}$ when $N$ is not prime. Theorem \ref{main:1} has the following corollary:

\begin{cor}
Suppose that $N$ has a prime divisor $p \ge 3$ such that $N \le p \sqrt{(p-1)/2}$. 
Then for any cusp $e \in X_{N, \lm}$ and $\lm \in \mathbb{P}^1 \setminus \{0, 1, \infty\}$, we have 
$$\Phi_1(\Delta_{{\rm GKS}, e}(X_{N, \lm})) \neq 0, $$
where $\Phi_1$ is the Abel-Jacobi map $\ch_1(X_{N,\lambda}^3)_{\rm hom}\rightarrow J_3(X_{N,\lm})$.
\end{cor}

\begin{proof}
Consider the map
$$f \colon X_{N, \lm} \to X_{p, \lm}; \quad ([x_1 :x_2], [y_1:y_2]) \mapsto ([x_1^{N/p} :x^{N/p}_2], [y^{N/p}_1:y^{N/p}_2]). $$
Then we have 
$$f_* \left( \Phi_1(\Delta_{{\rm GKS}, e}(X_{N, \lm}))\right)=  (N/p)^2 \cdot \Phi_1(\Delta_{{\rm GKS}, f(e)}(X_{p, \lm})),$$
where $\Phi_1$ denotes the Abel-Jacobi maps for both $X_{N,\lm}$ and $X_{p,\lm}$. Since $(N/p)^2 \le (p-1)/2$, by Theorem \ref{main:1}, the right-hand side is nontrivial. Thus the assertion follows. 
\end{proof}

\section*{Acknowledgments}
The authors would like to express sincere gratitude to the anonymous referees for their many helpful comments. This article was written while the second author was visiting the University of Winnipeg. He would like to thank the Department of Mathematics and Statistics for their hospitality. The second author was supported by Waseda University Grant for Early Career Researchers (Project number: 2025E-041). This work was partially supported by the NSERC Discovery Grant ``Arithmetic and Geometry of Mixed Motives".

\end{document}